\documentclass[runningheads,11pt]{llncs}
\usepackage{amssymb,amsmath,amsfonts}
\usepackage{graphicx}
\usepackage{subfigmat}
  \usepackage{geometry}
  \newgeometry{vmargin={35mm}, hmargin={40mm}} 
\usepackage{color}
\usepackage{url}
\usepackage{latexsym} 
\usepackage{wasysym}
\usepackage{newlfont}
\usepackage{stmaryrd}
\newcommand{\jumptmp}[2]{#1\llbracket{#2}#1\rrbracket}
\newcommand{\jump}[1]{\jumptmp{}{#1}}

\usepackage{xspace}
\newcommand{\gismo}{{\fontfamily{phv}\fontshape{sc}\selectfont G\pmb{+}Smo}\xspace}
\newcommand{\tripplenorm}[1]{{\left\vert\kern-0.25ex\left\vert\kern-0.25ex\left\vert#1 
    \right\vert\kern-0.25ex\right\vert\kern-0.25ex\right\vert}}
\newcommand{\normal}{\mathbf{n}}
\numberwithin{equation}{section}

\newcommand{\tabref}[1]{Table~\ref{#1}}
\numberwithin{equation}{section}
    \urldef{\mailsa}\path|stephen.moore@ucc.edu.gh|
    \newcommand{\keywords}[1]{\par\addvspace\baselineskip  
     \noindent\keywordname\enspace\ignorespaces#1}

\begin{document}
\mainmatter 
\title{\normalsize{MULTIPATCH DISCONTINUOUS GALERKIN IGA FOR THE BIHARMONIC PROBLEM ON SURFACES}}
\titlerunning{dGIGA for the Biharmonic Problem on Surfaces}
\author{Stephen E. Moore}
\authorrunning{S.~E.~Moore}
\institute{Department of Mathematics,\\ University of Cape Coast, Ghana.\\
\email{stephen.moore@ucc.edu.gh}
}
\maketitle
\begin{abstract}
We present the analysis of interior penalty discontinuous Galerkin Isogeometric Analysis (dGIGA) for the biharmonic problem on orientable surfaces $\Omega \subset \mathbb{R}^3.$ Here, we consider a surface  consisting of several non-overlapping patches as typical in multipatch dGIGA. Due to the non-overlapping nature of the patches, we construct NURBS approximation spaces which are discontinuous across the patch interfaces via a penalty scheme. By an appropriate discrete norm, we present \textit{a priori} error estimates for the non-symmetric, symmetric and semi-symmetric interior penalty methods. Finally, we confirm our theoritical results with numerical experiments.
\end{abstract}
\,\,
\section{Introduction}
\label{sec:Introduction}
Partial Differential Equations (PDEs) defined on surfaces embedded in $\mathbb{R}^3$ arise in many fields of application including material science, fluid mechanics, electromagnetics, biology and image processing. Fourth-order partial differential equations (PDEs) are particularly important in several areas of applied mechanics, the theory of elasticity, mechanics of elastic plates, and the slow flow of viscous fluids, see e.g. \cite{TimoshenkoWoinowsky:1959a}. Some examples of physical flows modelled by fourth order PDE include fluids on the lungs \cite{Halpern:1998a}, ice formation \cite{Myers:2004a}, imaging \cite{Facundo:2004}, designing special curves on surfaces \cite{Hofer:2004a}, modeling of interfaces in multiphase fluid flows and the modeling of surface active agents (surfactants), see e.g. \cite{Wheeler:2010a,NeNiRuWh:2008a}.

In this article, we consider the fourth-order boundary value problem: 
find $u: \overline \Omega \rightarrow \mathbb{R}$ such that
\begin{align}
\label{eqn:modelproblem}
        \Delta_{\Omega}^2 u + u= f \quad \text{in} \quad \Omega, 
\end{align}
where $f $ is a square integrable load vector defined on a compact smooth and orientable surface $\Omega \subset \mathbb{R}^3$ with boundary $\partial \Omega$ consisting of Dirichlet $\Gamma_D$ and Neumann $\Gamma_N$ boundaries, i.e. $\partial \Omega = \Gamma_D \cup \Gamma_N$ with the following considtions 
\begin{align}
	   u = g_0, \quad \mathbf{n} \cdot \nabla_{\Omega} u = g_1, \, \, \text{on} \, \,  \Gamma_D, 
   \quad \Delta_{\Omega} u = g_2, \quad & \mathbf{n} \cdot \nabla_{\Omega} \Delta_{\Omega} u = g_3 \, \,  \text{ on } \, \, \Gamma_N,
\end{align}
where $ \mathbf{n}$ is the outward directed normal vector to the boundary $\partial \Omega.$ We define the bi-Laplacian operator $\Delta_{\Omega}^2 := \Delta_{\Omega} \Delta_{\Omega} $ with $\Delta_{\Omega}$ as the Laplace-Beltrami operator, and the boundary data $g_0, g_1, g_2$ and $g_3$ are smooth functions.

Surface PDEs are usually treated with surface FEM which is a very popular discretization method. However, it is well known that surface FEM has major drawbacks due to the discrete 
variational formulation of the PDE that is constructed on a triangulated surface 
which contains the finite elements space \cite{DziukElliott:2013a}. 
Several works concerning the discretization of fourth order PDEs on surfaces using surface FEM  including an application of interior penalty Galerkin (IPG) methods have been presented, see e.g.  \cite{LarssonLarson:2013a}; which is an extension of the second order PDEs version  \cite{DednerMadhavanStinner:2013a}.

Fourth order PDE require continuously differentiably $(i.e. C^1-)$ piece-wise polynomial basis functions which are known to be practically difficult to construct as well as computationally expensive. However, in recent times, a new approximation method has been proposed that has $(p-1)-$continuously differentiable basis functions i.e. $C^{(p-1)},$ with degree $p\geq 1$ which makes it ideal towards the approximation of higher order PDEs including the biharmonic problem. This method is known as the isogeometric analysis (IGA). Moreover, IGA uses the same class of basis functions for both representing the geometry of the domain and approximating the solution of the PDEs \cite{BazilevsBeiraoCottrellHughesSangalli:2006a,TagliabueDedeQuarteroni:2014a}. 

 Multipatch discontinuous Galerkin IGA has been introduced and analyzed for second order elliptic problems on surfaces with matching and non-matching meshes, see e.g. \cite{LangerMoore:2014a,LangerMantzaflarisMooreToulopoulos:2015b,NKBBB:2014a,Moore:2020a}. Here, the computational domain consists of several conforming non-overlapping subdomains. By applying interior penalty methodology, we construct discrete spaces on the patches allowing for discontinuity along the patch interfaces. In this regards, the results presented in this article is an extention of the dGIGA to the biharmonic problem \cite{Moore:2018a}. In this article, we will present \textit{a priori} error estimate for multipatch discontinuous Galerkin isogeometric analysis (dGIGA) for biharmonic problem on conforming patches with matching meshes on orientable surfaces.

We organize the article as follows; 
In Section~\ref{sec:preliminaries}, the function spaces, weak formulation and the isogeometric analysis framework, NURBS surfaces, geometrical mappings and isogeometric analysis. The derivation of the interior penalty discontinuous Galerkin scheme is presented in Section~\ref{sec:InteriorPenaltyeqn:variationalformulation}. Then, in Section~\ref{sec:AnalysisOfDiscretizationError}, we present a discrete NURBS space $V_h,$ the discrete norms $\|\cdot\|_h$ and $\|\cdot\|_{h,*}$ then subsequently prove the coercivity of the bilinear forms. The boundedness of the bilinear forms is asserted in a product space $V_{h,*} \times V_h,$ where we will need another discrete norm $\|\cdot\|_{h,*}$  defined on the vector space $V_{h,*}.$ By using the idea of equivalence of norms, we are able to present coercivity and boundedness on the norm $\|\cdot\|_{h,*}.$ The error analysis of the dGIGA scheme is presented in Section~\ref{sec:erroranalysisOfdgigadiscretization}. In Section~\ref{sec:NumericalResults}, we present and discuss numerical experiments to confirm our theoretical results. Finally, we draw some conclusions and discuss future works in Section~\ref{sec:Conclusion}.
\section{Preliminaries}
\label{sec:preliminaries}
Let the computational domain $\Omega$ be a compact smooth and oriented surface with boundary $\partial \Omega.$ 
We introduce the Sobolev space 
$H^s(\Omega) :=\{v \in L_2(\Omega) \, : \, D^{\alpha}_{\Omega} v \in L_2(\Omega), \, \, \text{for} \, \, 0 \leq |\alpha| \leq s  \}$, where $L_2(\Omega)$ denote the space of square integrable functions and let $\alpha = (\alpha_1,\ldots,\alpha_d)$ be a multi-index with non-negative integers $\alpha_1,\ldots,\alpha_d$, $|\alpha| = \alpha_1 +\ldots + \alpha_d,$ $D^{\alpha}_{\Omega} := \partial^{|\alpha|}/\partial x^\alpha,$  and associate with the sobolev space $H^s(\Omega)$ the norm 
 $\|v\|_{H^s(\Omega)} = \left( \sum_{0 \leq |\alpha| \leq s} \|D^\alpha_{\Omega} v\|^2_{L_2(\Omega)} \right)^{1/2}$ see, e.g. \cite{AdamsFournier:2008}. 

The weak variational formulation of the biharmonic problem \eqref{eqn:modelproblem} reads:
find $u \in V_D$ such that 
\begin{align}
\label{eqn:variationalformulation}
     a(u,v) & = \ell(v), \quad \forall v \in V_0,
\end{align}
where the bilinear and linear forms are given by 
\begin{align}
 a(u,v) & = \int_{\Omega} \Delta_{\Omega} u \Delta_{\Omega} v\, dx + u v\, dx \quad\text{and} \quad \notag \\ 
 \ell(v) & = \int_\Omega f v\, dx + \int_{\Gamma_N} ( \mathbf{n} \cdot \nabla_{\Omega} v g_2 + v g_3 )ds \,
\end{align}
and the hyperplane  and test space given by   
$V_{D} :=\{v \in H^2(\Omega) : v|_{\Gamma_D} = g_0,\quad \mathbf{n} \cdot \nabla_{\Omega} v|_{\Gamma_D} = g_1 \}$ and 
$V_{0} :=\{v \in H^2(\Omega) : v|_{\Gamma_D} = 0, \quad \mathbf{n} \cdot \nabla_{\Omega} v|_{\Gamma_D} = 0 \}.$ 

The existence and uniqueness of the variational problem \eqref{eqn:variationalformulation} 
follows the well-known Lax-Milgram lemma see e.g. \cite{Ciarlet:2002a}. 
\subsection{NURBS Geometrical Mapping and Surfaces}
\label{subsec:surfaces}
{\allowdisplaybreaks
For the parameter domain $\widehat{\Omega},$ we define a vector-valued independent variable in the parameter domain $\widehat{\Omega}$ by $\xi = (\xi_1,\xi_2) \in \mathbb{R}^2.$ By means of a smooth and invertible geometrical mapping $\mathrm{\Phi}$, the computational  domain $\Omega \subset \mathbb{R}^3$ is defined as
 \begin{equation}
\label{eqn:geometricalmapping}
   \mathrm{\Phi} : \widehat{\Omega} \rightarrow \Omega \subset \mathbb{R}^3, 
  \quad 
  \mathbf{\xi} \rightarrow x = \mathrm{\Phi}(\mathbf{\xi}),
\end{equation}
where $\widehat{\Omega} \subset \mathbb{R}^2$ is the parameter domain  as illustrated in Figure~\ref{fig:igamap}.

  \begin{figure}[th!]
  \centering
      \includegraphics[width=0.6\textwidth]{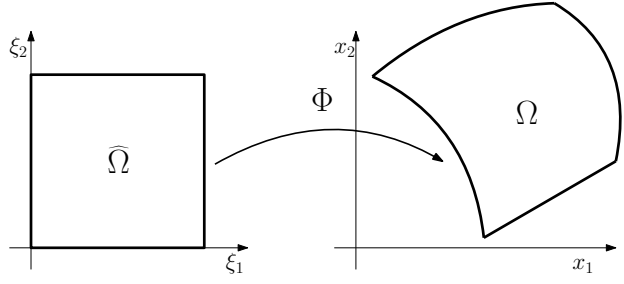}
      \caption{Illustration of the geometrical map $\Phi : \widehat{\Omega} \subset \mathbb{R}^2 \rightarrow \Omega \subset \mathbb{R}^3$ for a patch.}
      \label{fig:igamap}
 \end{figure}

We introduce briefly some important mathematical tools necessary for the analysis of surface PDEs. The following objects are obtained by means of the geometrical mapping \eqref{eqn:geometricalmapping} in the parameter domain. The Jacobian $\widehat{J},$ first fundamental form $\widehat{F}$ and the determinant $\widehat{g}$ of the geometrical mapping are respectively given by
\begin{align}
 \widehat{J} 
	&:=
	\left[ 
	\frac{\partial \mathrm{\Phi}_k}
 	{\partial \mathbf{\xi}_l}	
	\right] \in \mathbb{R}^{3 \times 2}, \quad k  = 1,2,3, \,l  = 1,2,\label{eqn:jacobian}\\ 
 \label{eqn:firstfundamentalanddeterminant}
 \widehat{F}(\xi) &= \left(\widehat{J}(\xi) \right)^{T}\widehat{J}(\xi) \in \mathbb{R}^{3 \times 3} \quad \text{and} \quad \widehat{g}(\xi)= \sqrt{\det \left({\widehat{F}(\xi) }\right)} \in \mathbb{R}.
\end{align}

Next, we present some differential operators by using notations in the parameter domain.
We consider a smooth function $\phi$ defined on the manifold $\Omega,$ by using the invertible geometrical mapping \eqref{eqn:geometricalmapping} to obtain   
\begin{equation}
 \label{eqn:invertiblemap}
 \phi ( x) = \widehat{\phi}(\xi) \circ \Phi^{-1}(x), \quad x \in \Omega,
\end{equation}
where $\widehat{\phi}(\xi) = \phi (\mathrm{\Phi}(\xi)).$
Using the gradient operator in the parameter space $\nabla \widehat{\phi}$, the tangential gradient of the manifold is given by 
\begin{equation}
\label{eqn:tangentialgradient}
  \nabla_{\Omega}\phi (x):= \widehat{J}(\mathbf{\xi})\widehat{F}^{-1}(\mathbf{\xi}) \nabla \widehat{\phi}(\mathbf{\xi}) \circ \Phi^{-1}(x).
\end{equation}
The divergence operator for the vector-valued function can be written as 
\begin{equation}
\label{eqn:divergenceoperator}
  \nabla_{\Omega} \cdot \varphi (x):= \dfrac{1}{\widehat{g}\mathbf{\xi})} \nabla \cdot \left[\widehat{g}(\mathbf{\xi}) \widehat{F}^{-1}(\xi)\widehat{J}^T(\xi) \widehat{\varphi}(\mathbf{\xi}) \right] \circ \Phi^{-1}(x) .
\end{equation}
The Laplace-Beltrami operator on the manifold $\Omega$ is defined for a twice continuously differentiable function $\phi :\Omega \rightarrow \mathbb{R}$ as 
\begin{equation}
\label{eqn:laplacebeltrami}
  \Delta_{\Omega}\phi (x) =  \dfrac{1}{\widehat{g}(\mathbf{\xi})} \nabla \cdot \left[\widehat{g}(\mathbf{\xi}) \widehat{F}^{-1}(\mathbf{\xi}) 
  \nabla \widehat{\phi}(\mathbf{\xi})\right] \circ \Phi^{-1}(x) .
\end{equation}
The surface gradient of the Laplace-Beltrami operator on the manifold $\Omega$ is defined for a thrice continuously differentiable function $\phi :\Omega \rightarrow \mathbb{R}$ as 
\begin{equation}
\label{eqn:gradientlaplacebeltrami}
 \nabla_\Omega \Delta_{\Omega}\phi (x) = \left[ \widehat{J}(\mathbf{\xi})\widehat{F}^{-1}(\mathbf{\xi}) \nabla  \left( \dfrac{1}{\widehat{g}(\mathbf{\xi})} \nabla \cdot \left[\widehat{g}(\mathbf{\xi}) \widehat{F}^{-1}(\mathbf{\xi}) 
  \nabla \widehat{\phi}(\mathbf{\xi})\right]\right)\right] \circ \Phi^{-1}(x) .
\end{equation}
The unit normal vector on the manifold $\Omega \subset \mathbb{R}^3$ is obtained by the geometrical mapping of 
\begin{equation}
 \label{eqn:unitnormalvector}
  \widehat{\normal}(\xi) := \dfrac{\widehat{t}_1(\xi) \times \widehat{t}_2(\xi)}{\| \widehat{t}_1(\xi) \times \widehat{t}_2(\xi)\|},
\end{equation}
where $\widehat{t}_l(\xi):= \partial \Phi(\xi) / \partial \xi_l$ is the tangent vector to a curve in $\mathbb{R}^3$ with $l = 1,2.$ The manifold $\Omega$ has a tangent plane at $\xi$ if the tangent vectors are linearly independent.

Finally, by means of the geometrical mapping \eqref{eqn:geometricalmapping}, we can write the Jacobian, first fundamental form and the determinant on the computational domain $\Omega$ as follows
\begin{align}
 \label{eqn:differentialoperatormanifold}
  J(x)  = \widehat{J}(\xi) \circ \mathrm{\Phi}^{-1}(x) ,  \quad 
  F(x) = \widehat{F}(\xi) \circ \mathrm{\Phi}^{-1}(x)  \quad \text{and} \quad 
  g(x) = \widehat{g}(\xi) \circ \mathrm{\Phi}^{-1}(x).
\end{align}
}
\subsection{B-Spline, NURBS and Isogeometric Analysis}
\label{subsec:bsplineandnurbs}
For a comprehensive understanding of isogeometric analysis, we refer the reader to  \cite{CottrellHughesBazilevs:2009a} and the references therein. However, for the purpose of completion, we present in this article briefly some vital information necessary for the formulation and discussion of multipatch isogeometric analysis.
For positive integers $p$ and $n,$ let us define a vector $\mathrm{\Xi} = \left \{0=\xi_1,\ldots,\xi_{n+p+1}=1\right\}$ with a non-decreasing sequence of real numbers in the parameter domain $\widehat{\Omega} = [0,1]$ called a knot vector on the unit interval $\widehat{\Omega} = [0,1].$
Given $\mathrm{\Xi}$ with $p \geq 1$ and $n$ as the number of basis functions, the univariate B-spline basis functions are defined by the Cox -de Boor recursion formula 
\begin{align}
 \widehat{B}_{i,0}(\xi)  & = \left\{
   \begin{aligned}
     & 1 & \text{if} \quad & \xi_{i} \leq \xi < \xi_{i+1},\\
     & 0 & \text{else}, & \notag \\
   \end{aligned}
    \right. \\
     \widehat{B}_{i,p}(\xi)  & = \frac{\xi-\xi_{i}}{\xi_{i+p}
     - \xi_{i}}\widehat{B}_{i,p-1}(\xi)
     +\frac{\xi_{i+p+1}-\xi}{\xi_{i+p+1}-\xi_{i+1}}\widehat{B}_{i+1,p-1}(\xi),
\end{align}
where any division by zero is defined to be zero. 
We note that a basis function of degree $p$ is $(p-m)$ times continuously differentiable across a knot value with the multiplicity $m$. For example, if all internal knots have the multiplicity $m = 1$, then B-splines of degree $p$ are globally $(p-1)-$continuously differentiable.

In general for higher-dimensional problems, the B-spline basis functions are tensor products of the univariate B-spline basis functions. We define tesor product basis functions as follows: 
let $ \mathrm{\Xi}_{\alpha} = \left \{\xi_{1,\alpha},\ldots,\xi_{n_{\alpha}+p_{\alpha}+1,\alpha}\right\}$ be the knot vectors for every direction $\alpha = 1,\ldots,d.$
Let $\mathbf{i}:= (i_1,\ldots,i_{d}),  \mathbf{p}:= (p_1,\ldots,p_{d})$ and the set 
$\overline{\mathcal{I}}=\{\mathbf{i} = (i_1,\ldots,i_{d}) : i_\alpha = 1,2, \ldots, n_\alpha; \;   \alpha = 1,2,  \ldots, d\} $ be multi-indicies.  
Then the tensor product B-spline basis functions are defined by
\begin{equation}
\label{eqn:chp2:multivariatebspline}
   \widehat{B}_{\mathbf{i},\mathbf{p}}(\xi) := \prod\limits_{\alpha=1}^{d}\widehat{B}_{i_\alpha,p_\alpha}(\xi_{\alpha}),
\end{equation}
where $\xi = (\xi_1,\ldots,\xi_{d}) \in \widehat{\Omega} = (0,1)^d.$ The univariate and multivariate B-spline basis functions are defined in the parametric domain by means of the corresponding B-spline basis functions $\{ \widehat{B}_{\mathbf{i},\mathbf{p}} \}_{\mathbf{i} \in \overline{\mathcal{I}}}.$ 

The distinct values $\xi_i, i=1,\ldots,n$ of the knot vectors $\mathrm{\Xi}$ provides a partition of $(0,1)^d$ creating a mesh $\widehat{\mathcal{K}}_h$ in the parameter domain where $\widehat{K}$ is a mesh element. 
The computational domain is described by means of a geometrical mapping $\mathbf{\Phi}$ such that  $\Omega = \mathbf{\Phi}(\widehat{\Omega})$ and 
\begin{align}
  \mathbf{\Phi}(\xi) := \sum_{\mathbf{i} \in \overline{\mathcal{I}}} C_\mathbf{i} \widehat{B}_{\mathbf{i},\mathbf{p}}(\xi),
\end{align}
where $C_\mathbf{i} $ are the control points. 
Next, we describe NURBS bassis functions. These basis functions are usually prefered in industry due to their ability to exactly represent most shapes and particularly conic families. 
The NURBS basis functions are obtained from the B-spline basis functions by means of a geometrical mapping as follows
\begin{align}
  \mathbf{\Phi}(\xi) :=  \sum_{i=1}^n \widehat{N}_{\mathbf{i},\mathbf{p}}(\xi)\mathbf{P}_i \quad \text{and} \quad 
  \widehat{N}_{\mathbf{i},\mathbf{p}}(\xi) :=  \dfrac{w_i}{\sum_{\mathbf{j}=1}^n \widehat{B}_{\mathbf{j},\mathbf{p}}(\xi)}\widehat{B}_{\mathbf{i},\mathbf{p}}(\xi),
\end{align}
where $\mathbf{P}_i,i=1,\ldots,n$ are the control points in the physical space and $\widehat{N}_{\mathbf{i},\mathbf{p}}(\xi)$ are the NURBS basis functions obtained by projective transformation of the B-spline basis functions with weight $w_i \in \mathbb{R}$ and $n$ the number of basis functions.
The basis functions in the computational domain are defined by means of the geometrical mapping as 
$N_{\mathbf{i},\mathbf{p}} := \widehat{N}_{\mathbf{i},\mathbf{p}} \circ \mathbf{\Phi}^{-1}$ and the discrete function spaces by 
\begin{equation}
 \label{eqn:discretefunctionspace}
 \mathbb{V}_h = \text{span}\{N_{\mathbf{i},\mathbf{p}}: \mathbf{i} \in  \overline{\mathcal{I}} \}.
\end{equation}
In several real life applications, the computational domain $\Omega$ is usually decomposed into $N$ non-overlapping sub-domains $\Omega_i$ called patches denoted by $\mathcal{T}_h := \{ \Omega_i \}_{i=1}^N$ such that $\overline{\Omega} = \bigcup_{i=1}^N \overline{\Omega}_i$ and $ \Omega_i \cap \Omega_j = \emptyset$ for $i \neq j.$ Each patch is the image of an associated geometrical mapping $\mathbf{\Phi}_i$ such that $\mathbf{\Phi}_i(\widehat{\Omega})= \Omega_i, i=1,\ldots,N,$ see Figure~\ref{fig:multipatchigamap}.
  \begin{figure}[h]
  \centering
      \includegraphics[width=0.8\textwidth]{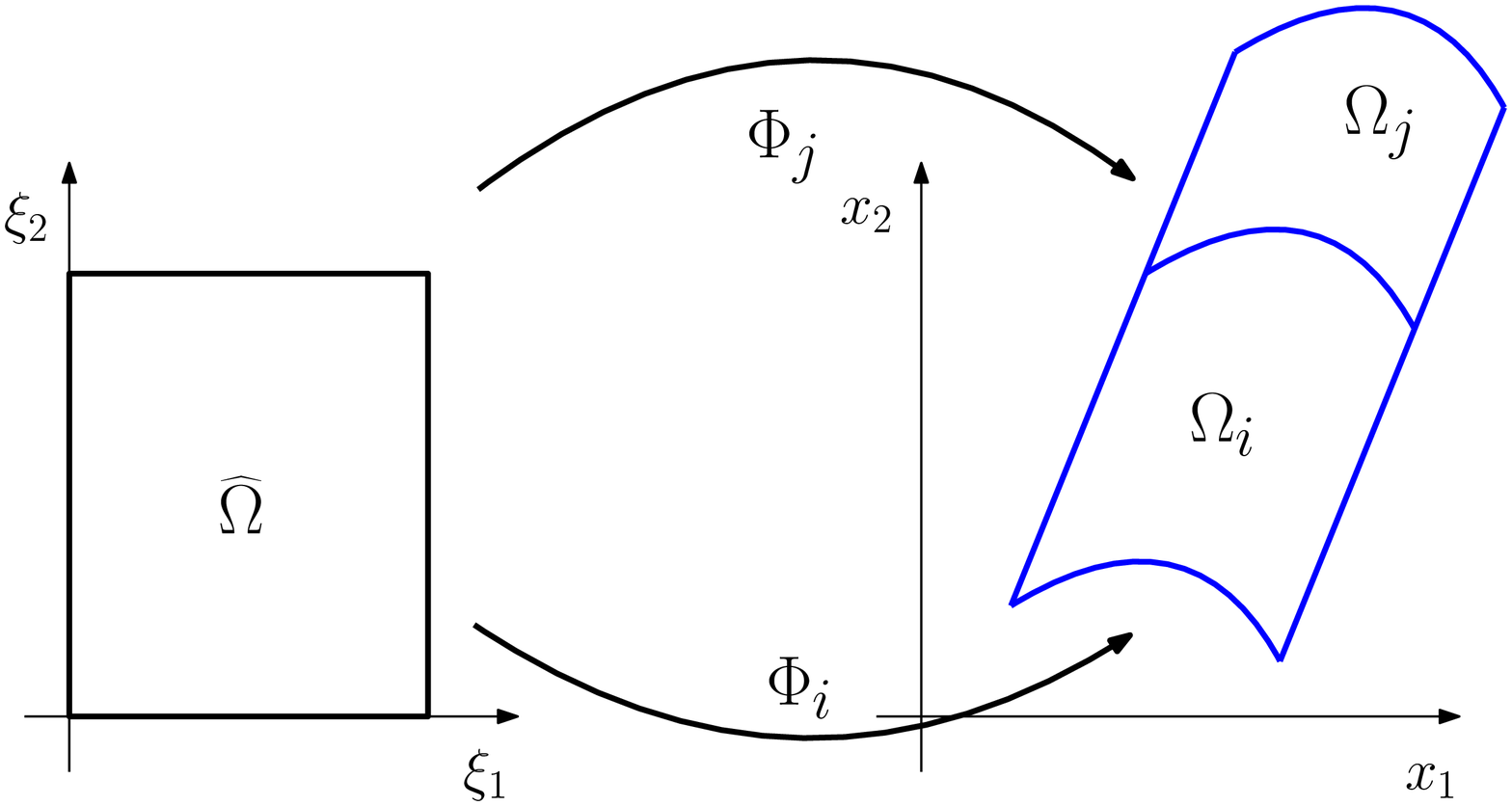}
      \caption{Illustration of the multi-patch isogeometric analysis map $\Phi_i (\widehat{\Omega}) =  \Omega_i$ and $\Phi_j (\widehat{\Omega}) =  \Omega_j, i 
\neq j.$}
      \label{fig:multipatchigamap}
 \end{figure}
We denote by $F_{ij} = \partial \Omega_i \cap \partial \Omega_j, i \neq j,$ the interior facets of two patches as illustrated in Figure~\ref{fig:multipatchiga}. The collection of all such interior facets is denoted by $\mathcal{F}_I,$ the set of Dirichlet facets $F_i$ on $\Gamma_D$ by $\mathcal{F}_D$ and the set of Neumann facets $F_i$ on $\Gamma_N$ by $\mathcal{F}_N.$ 
Furthermore, the collection of all internal, Dirichlet and Neumann facets is denoted by $\mathcal{F} := \mathcal{F}_I \cup \mathcal{F}_D \cup \mathcal{F}_N.$
\begin{figure}[th!]
  \centering
      \includegraphics[width=0.7\textwidth]{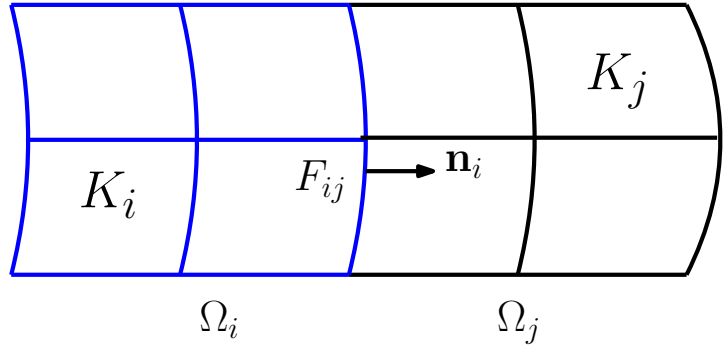}
      \caption{Illustration of the underlying mesh of the multi-patch isogeometric}
      \label{fig:multipatchiga}
\end{figure}

We assume that for each patch $\Omega_i, i=1,\ldots,N,$ the underlying mesh $\mathcal{K}_{h,i}$ is quasi-uniform i.e.  
\begin{align}
 \label{eqn:quasiuniformityassumption}
    h_K  \leq h_i \leq C_u h_K, \quad \text{for all} \quad K \in \mathcal{K}_{h,i}, \quad  i = 1,\ldots,N,
\end{align}
where $C_u \geq 1$ and $h_i = \max\{h_K, K \in \mathcal{K}_{h,i}\}$ is the mesh size of $\Omega_i$ and $h_K$ is the diameter of of the mesh element $K$. 
\section{Interior Penalty Variational Formulation}
\label{sec:InteriorPenaltyeqn:variationalformulation}
We introduce function spaces necessary for the derivation of interior penalty Galerkin schemes. Also, we assign to each patch $\Omega_i$ an integer $s_{ i}$ and collect them in the vector $\mathbf{s} = \{s_1,\ldots,s_N\}.$ We define the broken Sobolev space
\begin{equation}
 \label{eqn:brokensobolevspace}
 H^{\mathbf{s}}(\Omega,\mathcal{T}_h) :=\{v \in L_{2}(\Omega): v|_{\Omega_i} \in H^{s_i}(\Omega_i), \; i = 1,\ldots,N\},
\end{equation}
and the corresponding broken Sobolev norm and semi-norm
  \begin{align}
  \label{eqn:brokensobolevnorm}
    \|v \|_{H^{\mathbf{s}}(\Omega,\mathcal{T}_h)} := \left( \sum_{i=1}^N \|v \|^2_{H^{s_i}(\Omega_i)}\right)^{1/2} \quad \text{and} \quad 
     |v |_{H^{\mathbf{s}}(\Omega,\mathcal{T}_h)} := \left( \sum_{i=1}^N   |v |^2_{H^{s_i}(\Omega_i)}\right)^{1/2},
  \end{align} 
  respectively.
We define the jump and average of the normal derivatives across the interior facets $F_{ij} \in \mathcal{F}_I$ of $v \in H^{2}(\Omega,\mathcal{T}_h)$ by
\begin{align}
\label{eqn:jumpandaverage}
 \llbracket \nabla_{\Omega} v \rrbracket := \normal_i \cdot \nabla_\Omega v_i + \normal_j \cdot \nabla_\Omega v_j , \quad and \quad  \{\nabla_\Omega  v \} := \frac{1}{2}\left(\normal_i \cdot \nabla_\Omega v_i + \normal_j \cdot \nabla_\Omega v_j \right), 
 \end{align}
whereas the jump and average functions on the Dirichlet facets $F_i \in  \mathcal{F}_D$ are given by
 \begin{align}
 \label{eqn:dirichletjumpandaverage}
 \llbracket \nabla_\Omega v \rrbracket :=  \normal_i \cdot \nabla_\Omega v_i,  \quad and \quad  \{ \nabla_\Omega v \} :=  \normal_i \cdot \nabla_\Omega v_i. 
  \end{align} 
The following equality on the interior facets $F_{ij} \in \mathcal{F}_I$ is obtained by using the definitions of jumps and averages as 
 \begin{align}
 \label{eqn:dgmageicformula}
    \jump{ u v } & = \{a\} \jump{ b } + \{b\} \jump{a}, \quad \forall a,b \in \mathbb{R}.
 \end{align}
The interior penalty variational scheme reads: find $u \in H^4(\Omega,\mathcal{T}_h)$ such that,
\begin{equation}
\label{eqn:dgvariationalform}
	a_h (u,v) = \ell_h(v), \quad \forall v \in H^4(\Omega,\mathcal{T}_h),
\end{equation}
where the bilinear form is given by 
\begin{align}
\label{eqn:dgbilinearform}
	a_h(u,v) &= \sum_{i=1}^N \int_{\Omega_i} ( \Delta_{\Omega} u \Delta_{\Omega} v + u v ) \, dx \nonumber \\
	&- \sum_{F \in \mathcal{F}} \int_{F} \llbracket \nabla_{\Omega} v \rrbracket \{ \Delta_{\Omega} u\} \, ds 
	- \beta_0 \sum_{F \in \mathcal{F}} \int_{F} \llbracket \nabla_{\Omega} u \rrbracket \{ \Delta_{\Omega} v\} \, ds \nonumber \\
	&+ \sum_{F \in \mathcal{F}} \int_{F} \{\nabla_{\Omega} \Delta_{\Omega} u \} \llbracket v \rrbracket \, ds 
	+ \beta_1 \sum_{F \in \mathcal{F}} \int_{F} \{  \nabla_{\Omega} \Delta_{\Omega} v \} \llbracket u \rrbracket \, ds \nonumber \\ 
	&+ \sum_{F \in \mathcal{F}} \frac{\delta_1}{h_i^3} \int_{F} \llbracket u \rrbracket \llbracket v \rrbracket \, ds 
	+ \sum_{F \in \mathcal{F}} \frac{\delta_0}{h_i} \int_{F} \llbracket \nabla_{\Omega} u \rrbracket \llbracket \nabla_{\Omega} v \rrbracket  \, ds,
\end{align}
and the linear form reads as 
\begin{align}
\label{eqn:dglinearform}
	\ell_h(v) &= \sum_{i=1}^N \int_{\Omega_i} fv \, dx 
	          + \sum_{F_i \in \mathcal{F}_D} \int_{F_i} (\delta_1 v + \beta_1 \mathbf{n}_i \cdot \nabla_{\Omega} \Delta_{\Omega} v)g_0\,ds \notag \\
	         & + \sum_{F_i \in \mathcal{F}_D} \int_{F_i} (\delta_0 \mathbf{n}_i \cdot \nabla_{\Omega} v + \beta_0 \Delta_{\Omega} v)g_1\,ds 
	          + \sum_{F_i \in \mathcal{F}_N} \int_{F_i} (\mathbf{n}_i \cdot \nabla_{\Omega} v \,g_2  + v g_3 )\,ds.
\end{align}
The parameters $\beta_0, \beta_1 \in \{-1, 1\}$ and determine the interior penalty Galerkin (IPG) scheme. Here, we will describe the four main schemes a follows;
\begin{enumerate}
 \item $\beta_0=\beta_1=-1$ is the non-symmetric interior penalty Galerkin (NIPG)
 \item $\beta_0=\beta_1=1$  is the symmetric interior penalty Galerkin (SIGP)
 \item $\beta_0=-1, \beta_1=1$  is the semi-symmetric interior penalty Galerkin (SSIGP1)
 \item $\beta_0=1, \beta_1=-1$  is the semi-symmetric interior penalty Galerkin (SSIGP2)
\end{enumerate}

\begin{remark}
 \label{rem:penaltyparameter}
Although, there is no current literature on the choice of the penalty parameters for the bilinear form \eqref{eqn:dglinearform}, we choose $\delta_0=\delta_1=(p+1)(p+d)/d$ where $p$ is the degree of the NURBS and $d$ is the dimension of the computational surface i.e. $ d=3$ and $\Omega \subset \mathbb{R}^3.$ This choice of the penalty parameters yielded accurate simulations for the NIPG case presented in \cite{Moore:2018a}. 
\end{remark}
For the weak continuity of the fluxes on the interior facets, the exact solution $u$ must satisfy the following 
 \begin{align*}
   \llbracket u \rrbracket = 0, \llbracket \nabla_{\Omega} u \rrbracket = 0, 
   \llbracket \Delta_{\Omega} u \rrbracket = 0, \llbracket \nabla_{\Omega} \Delta_{\Omega} u \rrbracket = 0,\quad  \forall F \in \mathcal{F}. 
 \end{align*}
This enables us to show that the interior penalty Galerkin scheme is consistent, i.e. if $u \in  H^4(\Omega)$ is the solution of \eqref{eqn:variationalformulation}. Then $u$ is the solution to dGIGA variational identity \eqref{eqn:dgvariationalform}, see \cite{Moore:2018a} and \cite[Lemma~3]{SuliMozolevski:2004a}.
\section{Analysis of the dGIGA Scheme}
\label{sec:AnalysisOfDiscretizationError}
%
{\allowdisplaybreaks
In this section, we consider the existence and uniqueness of the bilinear form in the discrete setting. Thus, we require discrete spaces for the whole computational domain $\Omega.$ Since the domain consists of several subdomains or patches, we associate with each subdomain a discrete space as follows: 
Let us consider the B-Spline space $V_h \subset H^4(\Omega,\mathcal{T}_h)$ defined as
\begin{align}
 \label{eqn:nurbsspace}
  V_h := \{ v \in L_2(\Omega) : v|_{\Omega_i} \in \mathbb{V}_{h,i}, i=1,\ldots,N \},
\end{align}
where the B-Spline space $\mathbb{V}_{h,i}$ corresponds to the patch $\Omega_i, i=1,\ldots,N$ for B-splines of degree $p_i \geq 3.$
The discrete dGIGA scheme then reads as: find $u_h \in V_h$ such that 
\begin{align}
\label{eqn:discretedgvariationalform}
 a_h(u_h,v_h) = \ell(v_h), \quad \forall v_h \in V_h.
\end{align}
An immediate consequence of the consistency results as discussed in the latter part of section~\ref{sec:InteriorPenaltyeqn:variationalformulation} is the Galerkin orthogonality property, i.e., 
\begin{align}
\label{eqn:galerkinorthogonality}
    a_h(u - u_h,v_h) = 0, \quad \forall v_h \in V_h.
\end{align}
Next, for $v \in H^2(\Omega,\mathcal{T}_h),$ we show that the bilinear form 
$a_h(\cdot,\cdot),$ is coercive with respect to the following norm 
\begin{align}
 \label{eqn:dgnorm}
 \|v\|_h^2 := \sum_{i=1}^N \| \Delta_{\Omega} v \|^2_{L_2(\Omega_i)} 
            + \sum_{i=1}^N \| v \|_{L_{2}(\Omega_i)}^2
            + \sum_{F \in \mathcal{F}} \frac{\delta_1}{h_i^3} \|\llbracket v \rrbracket\|^2_{L_2(F)} 
            + \sum_{F \in \mathcal{F}} \frac{\delta_0}{h_i} \|\llbracket \nabla_{\Omega} v \rrbracket\|^2_{L_2(F)}.
\end{align}
\begin{remark}
For some function $v \in H^{2}(\Omega,\mathcal{T}_h),$ if $\|v\|_h = 0,$  then 
\begin{align}
 \label{eqn:interface}
  \Delta_{\Omega}  v & = 0 \quad \text{in} \quad \Omega_i \quad \text{for all} \quad i=1, \ldots, N \notag \\
  \llbracket v \rrbracket & = 0, \llbracket \nabla_{\Omega} v\rrbracket = 0 \quad on \quad F \in \mathcal{F}. 
\end{align}
Thus using the theory of elliptic interface problems $v = 0$ in the whole domain $\overline{\Omega}$ since it is a weak solution to \eqref{eqn:interface}. 
\end{remark}

The analysis of the dGIGA scheme requires the patch-wise inverse and trace inequalities given by the following lemmata, see \cite[chapter~2]{Moore:2017a}.
\begin{lemma}
\label{lem:inverseinequalities}
	Let $K_i \in \mathcal{K}_{h,i}, i =1,\ldots,N$. Then the inverse inequalities,
	\begin{align}
	\| \partial^j v \|_{L_2(\Omega_i)} &\leq C_{inv,1,u} h_i^{-1} \| \partial^{j-1} v \|_{L_2(\Omega_i)}, \label{eqn:inverseinequality-1}\\
	\|v \|_{L_2(\partial \Omega_i)} &\leq C_{inv,0,u} h_i^{-1/2} \| v \|_{L_2(\Omega_i)},
		\label{eqn:inverseinequality-2}
	\end{align}
hold for all $j \geq 1, v \in V_h,$ where $C_{inv,1,u}$ and $C_{inv,0,u}$ are positive constants, which are independent of $h_i$ and $\Omega_i$.
\end{lemma}
\begin{lemma}
 \label{lem:dgcoercivity}
 Let $a_h(\cdot,\cdot): V_h \times V_h \rightarrow \mathbb{R}$ be the discrete bilinear form defined in \eqref{eqn:discretedgvariationalform} with $\beta_0, \beta_1 \in \{-1,1\}$ and let
 $c_0$ and $c_1$ be nonnegative constants to be determined in the proof. Assume $\sigma_0 >0$ and $\sigma_1 >0$ and suppose that $\delta_0 \geq c_0$ and $\delta_1 \geq c_1.$ Then, there exists a positive constant $\mu_c$ such that 
 \begin{align}
  \label{eqn:dgcoercivity}
  a_h(v_h,v_h) \geq \mu_c \|v_h\|_h^2, \quad \forall v_h \in V_h.
 \end{align}
\end{lemma}
\begin{proof}
By setting $u_h=v_h$ in \eqref{eqn:dgbilinearform}, 
we proceed as follows
 \begin{align}
 \label{eqn:coercive-proof}
  	 a_h(v_h,v_h)   &= \sum_{i=1}^N \int_{\Omega_i} (\Delta_{\Omega} v_h)^2 \, dx 
	+ \sum_{F \in \mathcal{F}} \frac{\delta_1}{h_i^3} \int_{F} \llbracket v_h \rrbracket^2 \, ds  + \sum_{F \in \mathcal{F}} \frac{\delta_0}{h_i} \int_{F} \llbracket \nabla_{\Omega} v_h \rrbracket^2  \, ds \notag \\
	& - ( 1 +\beta_0)\sum_{F \in \mathcal{F}} \int_{F} \llbracket \nabla_{\Omega} v_h \rrbracket \{ \Delta_{\Omega} v_h\} \, ds
	  + (1+\beta_1 )\sum_{F \in \mathcal{F}} \int_{F} \llbracket v_h \rrbracket \{\nabla_{\Omega} \Delta_{\Omega}  v_h\} \, ds \nonumber \\
	& = \sum_{i=1}^N \|\Delta_{\Omega} v_h \|^2_{L_2(\Omega_i)} 
	+ \sum_{F \in \mathcal{F}} \frac{\delta_1}{h_i^3} \| \llbracket v_h \rrbracket \|^2_{L_2(F)}  + \sum_{F \in \mathcal{F}} \frac{\delta_0}{h_i} \|\llbracket \nabla_{\Omega} v_h \rrbracket\|^2_{L_2(F)} \notag \\ 
	& - (1 + \beta_0 )\sum_{F \in \mathcal{F}} \| \llbracket \nabla_{\Omega} v_h \rrbracket \|_{L_2(F)} \|\{ \Delta_{\Omega} v_h\} \|_{L_2(F)} \, \nonumber \\
	& + (1 + \beta_1 )\sum_{F \in \mathcal{F}} \| \llbracket v_h \rrbracket \|_{L_2(F)} \|\{\nabla_{\Omega} \Delta_{\Omega}  v_h\}\|_{L_2(F)}.
 \end{align}
Using Cauchy-Schwarz's inequality, Lemma~\ref{lem:inverseinequalities} and Young's inequality, we obtain
 \begin{align}
  & (1 + \beta_0 )\sum_{F \in \mathcal{F}} \| \llbracket \nabla_{\Omega} v_h \rrbracket \|_{L_2(F)} \|\{ \Delta_{\Omega} v_h\} \|_{L_2(F)} \notag \\
  & \leq  (1 + \beta_0) \left( \sum_{F \in \mathcal{F}} \dfrac{ C^2_{inv,0}}{2 h_i \sigma_0} \| \llbracket \nabla_{\Omega} v_h \rrbracket \|_{L_2(F)}^2 + \dfrac{\sigma_0}{2 } \sum_{i=1}^N \| \Delta_{\Omega} v_h \|_{L_2(\Omega_i)}^2  \right). \label{eqn:coerciveterm-1} \\
    & (1 + \beta_1 )\sum_{F \in \mathcal{F}} \| \llbracket  v_h \rrbracket \|_{L_2(F)} \|\{\nabla_{\Omega}  \Delta_{\Omega} v_h\} \|_{L_2(F)} \notag \\
  & \leq (1 + \beta_1) \left(\sum_{F \in \mathcal{F}}  \dfrac{ C^2_{inv,0} }{2 h_i^3\sigma_1 } \| \llbracket v_h \rrbracket \|_{L_2(F)}^2 + \dfrac{\sigma_1}{2 } \sum_{i=1}^N  \|\Delta_{\Omega} v_h \|_{L_2(\Omega_i)}^2  \right).  \label{eqn:coerciveterm-2}
 \end{align}
Substituting \eqref{eqn:coerciveterm-1} and \eqref{eqn:coerciveterm-2} into \eqref{eqn:coercive-proof} yields 
\begin{align}
 \label{eqn:coerciveform}
  a_h(v_h,v_h) & \geq \sum_{i=1}^N \bigg( 1 - \frac{\sigma_0}{2}(1+\beta_0) - \frac{\sigma_1}{2}(1+\beta_1) \bigg)\|\Delta_{\Omega}\|_{L_2(\Omega_i)}^2 \nonumber \\
  & + \sum_{F \in \mathcal{F}} \bigg(\delta_0 - \frac{ C^2_{inv,0,1}}{2\sigma_0}(1+\beta_0) \bigg) \| \llbracket \nabla_{\Omega} v_h \rrbracket \|_{L_2(F)}^2 \nonumber \\
  & + \sum_{F \in \mathcal{F}} \bigg(\delta_1 - \frac{ C^2_{inv,0}}{2\sigma_1}(1+\beta_1) \bigg) \| \llbracket v_h \rrbracket \|_{L_2(F)}^2.
\end{align}
The ellipticity constant is given by
\begin{equation*}
 \mu_c = 1 - \frac{\sigma_0}{2}(1+\beta_0) - \frac{\sigma_1}{2}(1+\beta_1),
\end{equation*}
and determined by the choices of $\sigma_0$ and $\sigma_1$ as well as $\beta_0$ and $\beta_1.$ The last two terms \eqref{eqn:coerciveform} hold for 
 $\delta_0 \geq  c_0 $ and $\delta_1 \geq c_1$ 
where 
\begin{equation*}
c_0 = \frac{(1+\beta_0) }{2\sigma_0}C^2_{inv,0,1} \quad \text{and} \quad 
c_1 = \frac{(1+\beta_1)}{2\sigma_1} C^2_{inv,0},
\end{equation*}
which completes the proof.
\qed
\end{proof}
Finally, we obtain the coercivity of the bilinear forms and the corresponding penalty parameters $\delta_0$ and $\delta_1$ in the next theorem.
\begin{theorem}
\label{thm:coercivity}
 Let $a_h(\cdot,\cdot)$ be the discrete bilinear form defined in \eqref{eqn:discretedgvariationalform} with positive constants $\delta_0$ and $\delta_1$
 and let $c_0$ and $c_1$ be nonnegative constants as in Lemma~\ref{lem:dgcoercivity}. Suppose that for 
Then there exists a positive constant $\mu_c$ such that 
 \begin{align*}
  a_h(v_h,v_h) \geq \mu_c \|v_h\|_h^2, \quad \forall v_h \in V_h.
 \end{align*}
\end{theorem}
\begin{proof}
 The proof follows from Lemma~\ref{lem:dgcoercivity} and the method is determined by the choice of $\beta_0$ and $\beta_1.$ \qed
\end{proof}

From Theorem~\ref{thm:coercivity}, we obtain the uniqueness of the solution of the discrete variational problem \eqref{eqn:discretedgvariationalform}. Since the discrete variational problem is in the finite dimensional space $V_h,$ the uniqueness therefore yields the existence of 
the solution $u_h \in V_h$ of \eqref{eqn:discretedgvariationalform}. 
\begin{lemma}
	\label{lem:scaledtraceinequality}
	Let $K \in \mathcal{K}_{h,i},i = 1,\ldots,N$ and $\widehat K = \Phi_i^{-1}(K)$. Then the scaled trace inequality
	\begin{align}
	\label{eqn:scaledtraceinequality}
		\|v\|_{L_2(\partial \Omega_i)} \leq C_{tr,u}h_i^{-1/2} \left( \|v\|_{L_2(\Omega_i)} + h_i |v|_{H^1(\Omega_i)} \right),
	\end{align}
holds for all $v \in H^1(\Omega_i),$ where $h_i$ denotes the global mesh size of patch $\Omega_i$ in the physical domain, and $C_{tr,u}$ is a positive constant that only depends on the quasi-uniformity and shape regularity of the mapping $\Phi_i$.
\end{lemma}
To enable us derive uniform boundedness of the bilinear form $a_h(\cdot,\cdot):V_{h,*} \times V_h \rightarrow \mathbb{R},$ where $V_{h,*} := V_D \cap H^\mathbf{s}(\Omega,\mathcal{T}_h) +  V_h$ with $\mathbf{s} \geq 4$ and equipped with the norm
\begin{align}
 \label{eqn:dgnorm*}
 \|v\|_{h,*}^2 = \|v\|_h^2  
               + \sum_{F \in \mathcal{F}} \frac{h_i^3}{\delta_1} \|\{ \nabla_{\Omega} \Delta_{\Omega} v \}\|^2_{L_2(F)} 
               + \sum_{F \in \mathcal{F}} \frac{h_i}{\delta_0} \|\{ \Delta_{\Omega} v \}\|^2_{L_2(F)}.
\end{align}
Indeed, the norm \eqref{eqn:dgnorm} is also a norm on $H^4(\Omega,\mathcal{T}_h)$ since $H^4(\Omega,\mathcal{T}_h) \subset H^2(\Omega,\mathcal{T}_h).$ 
\begin{lemma}
 \label{lem:dgboundedness}
 Let $a_h(\cdot,\cdot) : V_{h,*} \times V_{h}$ be the bilinear form defined in \eqref{eqn:dgbilinearform} with $\beta_0, \beta_1 \in \{-1,1\}$ and $\delta_0, \delta_1 > 0.$ Then there exists a positive constant $\mu_b,$ such that 
 \begin{align}
  \label{eqn:dgboundedness}
  | a_h(u,v_h)| \leq \mu_b \|u\|_{h,*} \|v_h\|_h, \quad \forall u \in V_{h,*},  v_h \in V_h.
 \end{align}
\end{lemma} 
\begin{proof}
The proof follows by using the he Cauchy-Schwarz inequality to estimate the terms in the bilinear form \eqref{eqn:dgbilinearform}. However, concerning the concerning the third term, we apply the inverse inequality \eqref{eqn:inverseinequality-2} for $v_h \in V_h$ as follows
  \begin{align*}
 	 \bigg| \sum_{F \in \mathcal{F}} \int_{F} \{ \Delta_{\Omega} v_h \} \llbracket \nabla_{\Omega} u \rrbracket \, ds \bigg|
 	&\leq \left( \sum_{i=1}^N \frac{C_{inv,0,u}^2}{\delta_0} \|  \Delta_{\Omega} v_h\|^2_{L_2(\Omega_i)} \right)^{\frac{1}{2}} \left( \sum_{F \in \mathcal{F}} \frac{\delta_0}{h_i} \| \llbracket \nabla_{\Omega} u \rrbracket \|^2_{L_2(F)} \right)^{\frac{1}{2}}.
 \end{align*}
The fifth term is also estimated by using the inverse inequalities \eqref{eqn:inverseinequality-1} and \eqref{eqn:inverseinequality-2} for $v_h \in V_h,$ to obtain 
   \begin{align}
 	\bigg| \sum_{F \in \mathcal{F}} \int_{F} \{ \nabla_{\Omega} \Delta_{\Omega} v_h \} \llbracket u \rrbracket  \, ds \bigg|  
 	& \leq \left( \sum_{i=1}^N\frac{C_{inv,0,1}^2}{\delta_1} \| \Delta_{\Omega} v_h \|^2_{L_2(\Omega_i)} \right)^{\frac{1}{2}} \left( \sum_{F \in \mathcal{F}} \int_{F}  \frac{\delta_1}{h_i^3} \| \llbracket u \rrbracket\|^2_{L_2(F)} \right)^{\frac{1}{2}},
 \end{align}
 where $C_{inv,0,1}^2 = C_{inv,0,u}^2C_{inv,1,u}^2.$
By puttting all the terms together and using Cauchy-Schwarz's inequality, we complete the proof 
 with the boundedness constant given as 
 $\mu_b  = 2 \sqrt{\max\{ 1, (1 + C_{inv,0,1}^2|\beta_1|/\delta_1 + C^2_{inv,0,u}|\beta_0|/\delta_0) \} }.$ \qed
\end{proof}
It is also possible to show the existence and uniqueness results for the norm $\|\cdot \|_{h,*}$ due to the uniform equivalence of norms on $V_h.$
\begin{lemma}
 \label{lem:normequivalence}
 The norms $\|\cdot\|_h$ and $\|\cdot\|_{h,*}$ are uniformly equivalent on the discrete space $V_h$ such that 
 \begin{equation}
  \label{eqn:normequivalence}
   C^* \|v_h \|_{h,*} \leq \|v_h \|_h \leq \|v_h \|_{h,*}, \quad \forall v_h \in V_h,
 \end{equation}
 where $C^*$ is a mesh independent positive constant.
\end{lemma}
\begin{proof}
 The upper bound follows immediately. The lower bound is derived by applying the inverse inequalities of Lemma~\ref{lem:inverseinequalities} where we coplete the proof with $C^* = \left(1 + C_{inv,0,1}^2/\delta_0 +C_{inv,0,u}^2/\delta_1 \right)^{-1}$ and $C_{inv,0,1}^2 = C_{inv,0,u}^2C_{inv,1,u}^2$. \qed
\end{proof}
Due to Lemma~\ref{eqn:normequivalence}, we can derive coercivity and boundedness results for the bilinear form $a_h(\cdot,\cdot)$ in the norm $\|\cdot\|_{h,*}.$ By using the results from Theorem~\ref{thm:coercivity}, the coercivity on $V_h$ yields
 \begin{equation}
 \label{eqn:coercivity-2}
  a_h(v_h,v_h) \geq \widetilde{\mu}_c \|v_h\|_{h,*}^2, \quad \forall v_h \in V_h,
 \end{equation}
where $\widetilde{\mu}_c = C^* \mu_c$ is a nonnegative constant independent of $h.$ 
Also, the boundedness of the bilinear form following from Lemma~\ref{lem:dgboundedness} is given by
\begin{align}
  \label{eqn:dgboundedness-2}
  | a_h(u_h,v_h)| \leq \tilde{\mu}_b \|u_h\|_{h,*} \|v_h\|_{h,*}, \quad \forall (u, v_h) \in V_{h,*} \times V_h,
 \end{align}
 with $\tilde{\mu}_b$ independent of $h.$ 
}
\section{Error Estimates for dGIGA discretization scheme}
\label{sec:erroranalysisOfdgigadiscretization}
To obtain \textit{a priori} error estimates in both $\|\cdot\|_h$ and $\|\cdot\|_{h,*}$ norms, we require approximation estimates by means of a quasi-intepolant. We denote by $\Pi_h : L_2(\Omega_i) \rightarrow V_{h,i}$ such a quasi-interpolant that yields optimal approximation results for each patch or subdomain $\Omega_i, i=1,\ldots,N,$ see \cite{BazilevsBeiraoCottrellHughesSangalli:2006a,DaVeigaBuffaSangalliVazquez:2014a} for the proof. 
\begin{lemma}
\label{lem:localerrorestimate}
Let $l_i$ and $s_i$ be integers with $0 \leq l_i \leq s_i \leq p_i+1$ 
 and $K \in \mathcal{K}_{h,i}$. Then there exist an interpolant $\Pi_h v \in V_{h,i}$ 
 for all $ v \in H^{s_i}(\Omega_i)$ and a constant $C_s > 0$ such that 
 the following inequality holds 
 \begin{align}
 \label{eqn:localerrorestimate}
    \sum_{K \in \mathcal{K}_{h,i}}|v - \Pi_{h}v|^{2}_{H^{l}(K)} 
      & \leq C_s  h_i^{2(s_i-l_i)} \|v \|^2_{H^{s_i}(\Omega_i)}, 
 \end{align} 
where $h_i$ is the mesh size in the physical domain, and $p$ denotes the underlying polynomial degree of the B-spline or NURBS.
\end{lemma}
If the multiplicity of the inner knots is not larger than $p_i + 1 - l_i$ and $\Pi_h v \in V_h \cap H^{l_i}(\Omega_i)$ for each patch $\Omega_i, i=1,\ldots,N,$ then the local estimate \eqref{eqn:localerrorestimate} yields a global estimate.
\begin{proposition}
\label{prop:globalerrorestimate}
 Let us assume that the multiplicity of the inner knots is not larger than 
 $p_i + 1 - l_i.$  Given the integers $l_i$ and $s_i$ such that 
 $0\leq l_i \leq s_i \leq p_i+1,$ there exist a positive constant $C_s$ such that
  for a function $v \in H^{s_i}(\Omega_i)$ 
 \begin{equation}
 \label{eqn:globalerrorestimate}
   |v - \Pi_{h}v|_{H^{l_i}(\Omega_i)}  
   \leq C_s h_i^{(s_i-l_i)} \|v \|_{H^{s_i}(\Omega_i)}, 
 \end{equation} 
where $h_i$ denotes the maximum mesh-size parameter in the physical domain and the generic constant $C_s$ only depends on $l_i,s_i$ and $p_i$, the shape regularity of the physical domain $\Omega_i$ described by the mapping $\Phi$ and, in particular, $\nabla_{\Omega} \mathrm{\Phi}.$
\end{proposition}
\begin{proof}
See \cite[Proposition~3.2]{TagliabueDedeQuarteroni:2014a}. \qed
\end{proof}
%
We consider that the quasi-interpolant is the same for each patch, 
i.e. $\Pi_h : V_D \cap H^{\mathbf{s}}(\Omega,\mathcal{T}_h) \rightarrow V_h$ with $\mathbf{s} \geq 4.$ 
\begin{lemma}
\label{lem:interfaceapproximationerror}
Let $v \in V_D \cap H^{\mathbf{s}}(\Omega, \mathcal{T}_h)$ with a positive integer
$\mathbf{s} \geq 4$ and let $F \in \mathcal{F}_I \cup \mathcal{F}_D$ be the facets. Also, let $\delta_0$ and $\delta_1$ be chosen as in Theorem~\ref{thm:coercivity}. By assuming quasi-uniform meshes, then there exists a quasi-interpolant $\Pi_h$ such that $\Pi_{h}v \in V_h$ and
the following estimates hold;
  \begin{align}
       \|\nabla_{\Omega} ^q( v - \Pi_{h} v) \|_{L_2(\partial \Omega_i)}^{2} &  \leq c_2 h^{2(r_i-q)-1}_{i} \| v \|^{2}_{H^{r_i}(\Omega_i)},  \label{eqn:Interface0} \\
     \sum_{F \in \mathcal{F}}\frac{\delta_1}{h_i^3} \|\llbracket v - \Pi_{h} v \rrbracket \|_{L_2(F)}^{2} &  \leq c_3 \sum_{i=1}^N h^{2(r_i-2)}_{i}\| v \|^{2}_{H^{r_i}(\Omega_i)}, \label{eqn:InterfaceI} \\
    \sum_{F \in \mathcal{F}} \frac{\delta_0}{h_i}  \| \llbracket \nabla_{\Omega} ( v - \Pi_{h} v) \rrbracket  \|_{L_2(F)}^{2} & \leq c_4 \sum_{i=1}^N h^{2(r_i-2)}_{i}\| v \|^{2}_{H^{r_i}(\Omega_i)}, \label{eqn:InterfaceII}\\
     \sum_{F \in \mathcal{F}} \frac{h_i}{\delta_0} \| \{ \Delta_{\Omega} ( v - \Pi_{h} v) \} \|_{L_2(F)}^{2} &  \leq c_5  \sum_{i=1}^Nh^{2(r_i-2)}_{i}\| v \|^{2}_{H^{r_i}(\Omega_i)}, \label{eqn:InterfaceIII}\\
    \sum_{F \in \mathcal{F}}\frac{h_i^3}{\delta_1} \|\{ \nabla_{\Omega}  \Delta_{\Omega}  ( v - \Pi_{h} v)\}  \|_{L_2(F)}^{2} & \leq c_6 \sum_{i=1}^Nh^{2(r_i-2)}_{i}\| v \|^{2}_{H^{r_i}(\Omega_i)}, \label{eqn:InterfaceIV}
  \end{align}
where $ q$ is a positive integer, $r_i =\min\{s_i,p_i + 1\}$ and the generic constants $c_2, c_3,c_4,c_5$ and $c_6$ are independent of the mesh size. 
\end{lemma}
\begin{proof}
 By using the trace inequality \eqref{eqn:scaledtraceinequality} and the approximation estimate \eqref{eqn:globalerrorestimate}, we obtain the proof, see \cite{Moore:2018a}.
	\qed	    
\end{proof}
In the next lemma, we present estimates in the discrete norms $\|\cdot\|_h$ and $\|\cdot\|_{h,*}$ necessary for deriving the \textit{a priori} error estimates. 
\begin{lemma}
\label{lem:dgInterpolantionbound}
Let $v \in V_D \cap H^{\textbf{s}}(\Omega,\mathcal{T}_h)$ for $\mathbf{s} :=(s_1,s_2,\ldots,s_N) \geq 4$ and $\Pi_{h}v \in V_{h}$  be a projection. 
 Then, for $\mathbf{p}:=(p_1,p_2,\ldots,p_N)  \geq 3,$ we have 
 \begin{align}
  \|v- \Pi_h v \|_h^2 & \leq c_7 \sum_{i=1}^{N} h^{2(r_i-2)}_{i}\| u\|^{2}_{H^{r_i}(\Omega_i)},  \label{eqn:dgnormInterpolantionbound} \\
  \|v- \Pi_h v \|_{h,*}^2 & \leq  c_8 \sum_{i=1}^{N}h^{2(r_i-2)}_{i}\| u\|^{2}_{H^{r_i}(\Omega_i)} \label{eqn:dgnormInterpolantionbound*},
 \end{align}
  where $r_i := \min\{s_i, p_i+1\}$, $p_i$ is the degree of the B-spline 
  and the constants $c_7$ and $c_8$ are independent of mesh size $h_i.$ 
\end{lemma}
\begin{proof}
The proof follows by using Lemma~\ref{lem:interfaceapproximationerror} together with the definitions of the norms \eqref{eqn:dgnorm} and \eqref{eqn:dgnorm*}.
 \qed
\end{proof}
Finally, we present the \textit{a priori} error estimate in the norm $\|\cdot \|_h$ and $\|\cdot \|_{h,*}$ as follows
\begin{theorem}
 \label{thm:dgerrorestimate}
  Let $u \in V_D \cap H^{\textbf{s}}(\Omega,\mathcal{T}_h)$ with $\mathbf{s} = \{s_i,i =1,\ldots,N \}, s_i\geq 4$ be the solution of \eqref{eqn:dgvariationalform} for non-negative real numbers $\delta_0$ and $\delta_1$ chosen as in Theorem~\ref{thm:coercivity}. Let $u_h \in V_h$ be the solution of \eqref{eqn:discretedgvariationalform}, then there exists $c > 0$ independent of $h_i$ and $N$ such that the following holds:
  \begin{align}
  \label{eqn:dgerrorestimate}
   \|u-u_h\|_h & \leq c \sum_{i=1}^{N}h^{r_i-2}_{i}\| u \|_{H^{r_i}(\Omega_i)} 
  \end{align}
  where $r_i =\min\{s_i,p_i+1\}.$ 
\end{theorem}
\begin{proof}
 Using the coercivity result of Lemma~\ref{lem:dgcoercivity}, the Galerkin orthogonality \eqref{eqn:galerkinorthogonality} and the boundedness of Lemma~\ref{lem:dgboundedness}, 
 we can derive the following estimates
\begin{align*}
  \mu_c\|\Pi_h u - u_h\|_h ^2 & \leq a_h(\Pi_h u - u_h,\Pi_h u - u_h)
				  = a_h(\Pi_h u - u,\Pi_h u - u_h) \\
				& \leq \mu_b \|\Pi_h u - u\|_{h,*}\| \Pi_h u - u_h\|_{h}.
\end{align*}
 Thus, we have
 \begin{align}
 \label{eqn:dgerrorintermediate}
   \mu_c\|\Pi_h u- u_h\|_h & \leq \mu_b \|\Pi_h u- u\|_{h,*}.
 \end{align}
 Using \eqref{eqn:dgerrorintermediate} and the estimates from  
 Lemma~\ref{lem:dgInterpolantionbound}, we obtain by using triangle inequality the 
 following estimate
 \begin{align}
   \|u-u_h\|_h &\leq \|u- \Pi_h u\|_h + \mu_b \|\Pi_h u- u\|_{h,*} \nonumber \\
   &\le \left( c_7^{1/2} + (\mu_b/\mu_c)  c_8^{1/2} \right)\sum_{i=1}^{N}h^{r_i-2}_{i}\| u\|_{H^{r_i}(\Omega_i)},
 \end{align}
 with $ c= (c_7^{1/2}+ (\mu_b/\mu_c) c_8^{1/2}).$
  \qed
\end{proof}
\begin{corollary}
\label{cor:dgerrorestimate}
 Following the hypothesis and assumptions of Theorem~\ref{thm:dgerrorestimate}, then there exists a constant $c_9$ such that the following estimate holds
   \begin{align}
  \label{eqn:dgerrorestimate-2}
   \|u-u_h\|_{h,*} & \leq c_9 \sum_{i=1}^{N}h^{r_i-2}_{i}\| u \|_{H^{r_i}(\Omega_i)},
  \end{align}
  where $r_i =\min\{s_i,p_i+1\}.$ 
\end{corollary}
\begin{proof}
 The proof follows similar argument as in Theorem~\ref{thm:dgerrorestimate} using 
 \eqref{eqn:coercivity-2} and \eqref{eqn:dgboundedness-2} with 
 $c_9=(c_7^{1/2}+ (\widetilde{\mu_b}/\widetilde{\mu_c}) c_8^{1/2}).$
 \qed
\end{proof}
The estimates from Theorem~\ref{thm:dgerrorestimate} yields a sharper bound since the estimate from corollary~\ref{cor:dgerrorestimate} relies on the norm equivalence i.e. Lemma~\ref{lem:normequivalence}.  
\section{Numerical Results}
\label{sec:NumericalResults}
In this section, we present numerical results for the model problem \eqref{eqn:modelproblem}. 
The numerical experiments are carried out in \gismo; an open source object-oriented simulation tool developed solely for IGA, see,  \cite{JuettlerLangerMantzaflarisMooreZulehner:2014a}.
The penalty parameters $\delta_0=\delta_1 = (p+1)(p+3)/3$ where $p$ is the NURBS degree, see Remark~\ref{rem:penaltyparameter}. We consider for the open surface, a quarter cylinder and for the closed surface, a torus as computational domains. Each of the surfaces considered consists of four $(4)$ patches with matching underlying meshes. The resulting linear system from the discrete dGIGA scheme \eqref{eqn:discretedgvariationalform} has been solved using the SuperLU on an Intel Core (TM) i5-4300 CPU $@1.90GHz.$ The convergence rate is computed using the formula $rate = \log_2 \left( e_{i}/e_{i+1} \right),$ where $e_{i+1} = \|u -u_{h,i+1}\|_h$ and $e_{i} = \|u -u_{h,i}\|_h$ to study the discrete solution of the model problem. 
\subsection{Open Surface}
\label{subsec:opensurface}
 We consider a Dirichlet biharmonic problem on an open surface $\Omega$ that is given by a quarter cylinder in the first quadrant i.e. $x \geq 0$ and $y \geq 0$ with unitary radius and height $L = 4.$ The computational domain $\Omega$ is decomposed into 4 patches, with each of 
 the patches having a unit height of one $(1)$ and depicted by a different color as seen on the 
left-hand side of Figure~\ref{fig:opensurface}. The knot vectors representing the geometry of each patch are given by $\Xi_{1}= \{0, 0, 0, 1, 1, 1\}$ and $\Xi_{2}= \{0, 0, 1, 1\}$ in the $\xi_1-$direction and $\xi_2-$direction respectively. The exact solution is chosen as $u(\phi, z) = \varrho g_{\phi,1}(\phi)g_{z}(z),$ where $\phi:= \arctan \left(\frac{x}{y}\right), g_{\phi,1}(\phi)=(1-\cos(\phi))(1-\sin(\phi))$ and $g_{z}(z)=\sin(\sigma \pi z/L)$ are in cylindrical coordinates $(r,\phi,z)$ which yields the source function
\begin{equation*}
 f(\phi,z) = \frac{\varrho g_z(z)\bigg(2\pi^4\sigma^4 + (\pi^2\sigma^2+4L^2)^2\sin(2\phi)-2(\pi^2\sigma^2+L^2)^2(\sin(\phi)+\cos(\phi))\bigg)}{2L^4}.
\end{equation*}
In the numerical experiments, we set $\sigma = 3, \varrho = 1/\left(3/2 -\sqrt{2} \right)$ and present the contours of the solution, see, Figure~\ref{fig:opensurface} (right).
The rate of convergence for all four $(4)$ dGIGA schemes with respect to the discrete norm $\|\cdot\|_h$ is presented in Table~\ref{tab:opensurface} by successive mesh refinement for NURBS degree $2 \leq p \leq 6.$ We observe the optimal convergence rate $\mathcal{O}(h^{(p-1)})$ as theoretically predicted in Theorem~\ref{thm:dgerrorestimate}. Since the discrete norm $\|\cdot \|_h$ is indeed a norm on $H^2(\Omega,\mathcal{T}_h),$ we require at least a continuously differentiable basis function i.e. $C^{p-1}, p \geq 2$ to obtain optimal convergence rates. 
    \begin{figure}[thb!]
     \centering
      \includegraphics[width=0.48\textwidth]{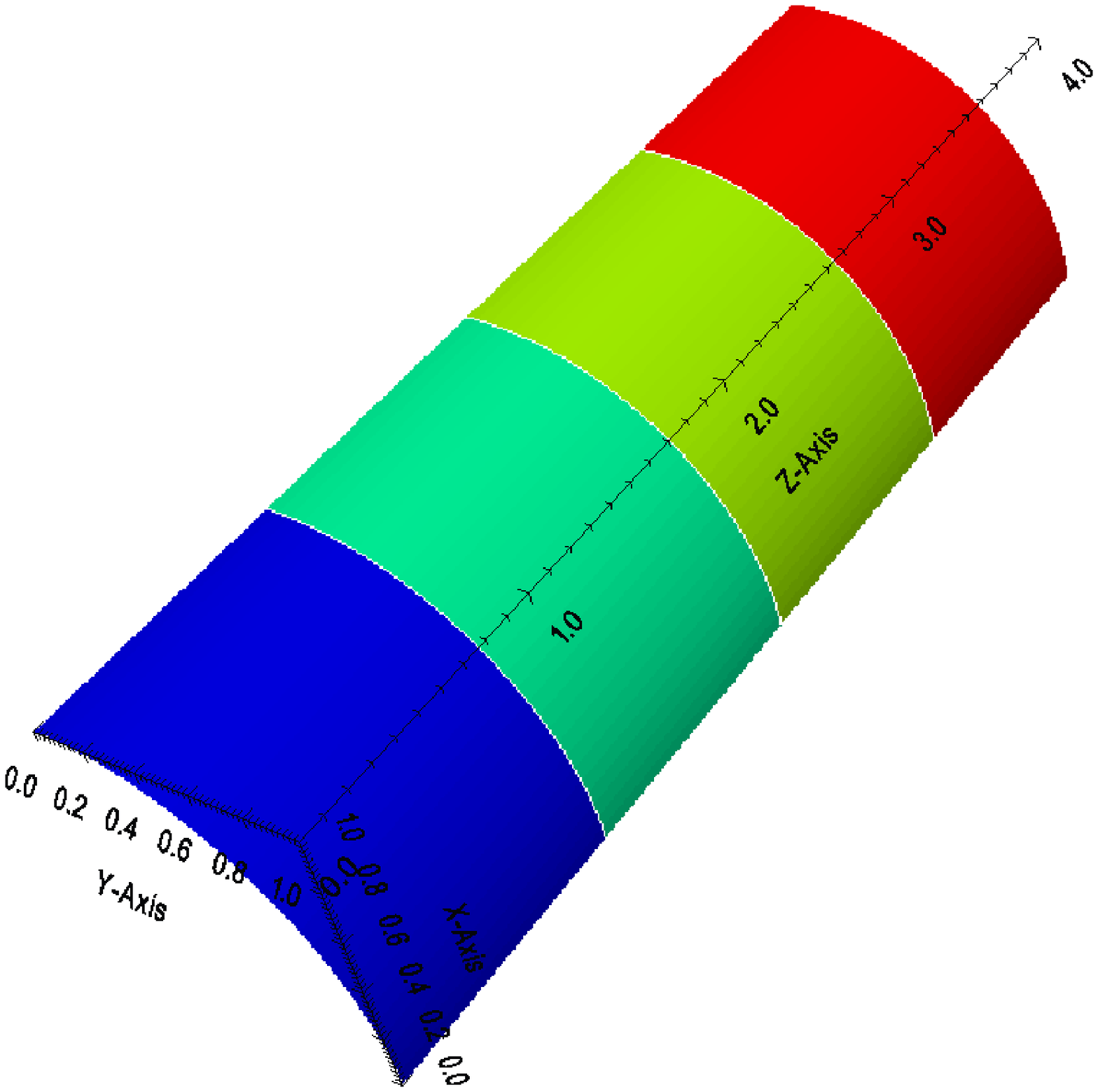}
      \includegraphics[width=0.48\textwidth]{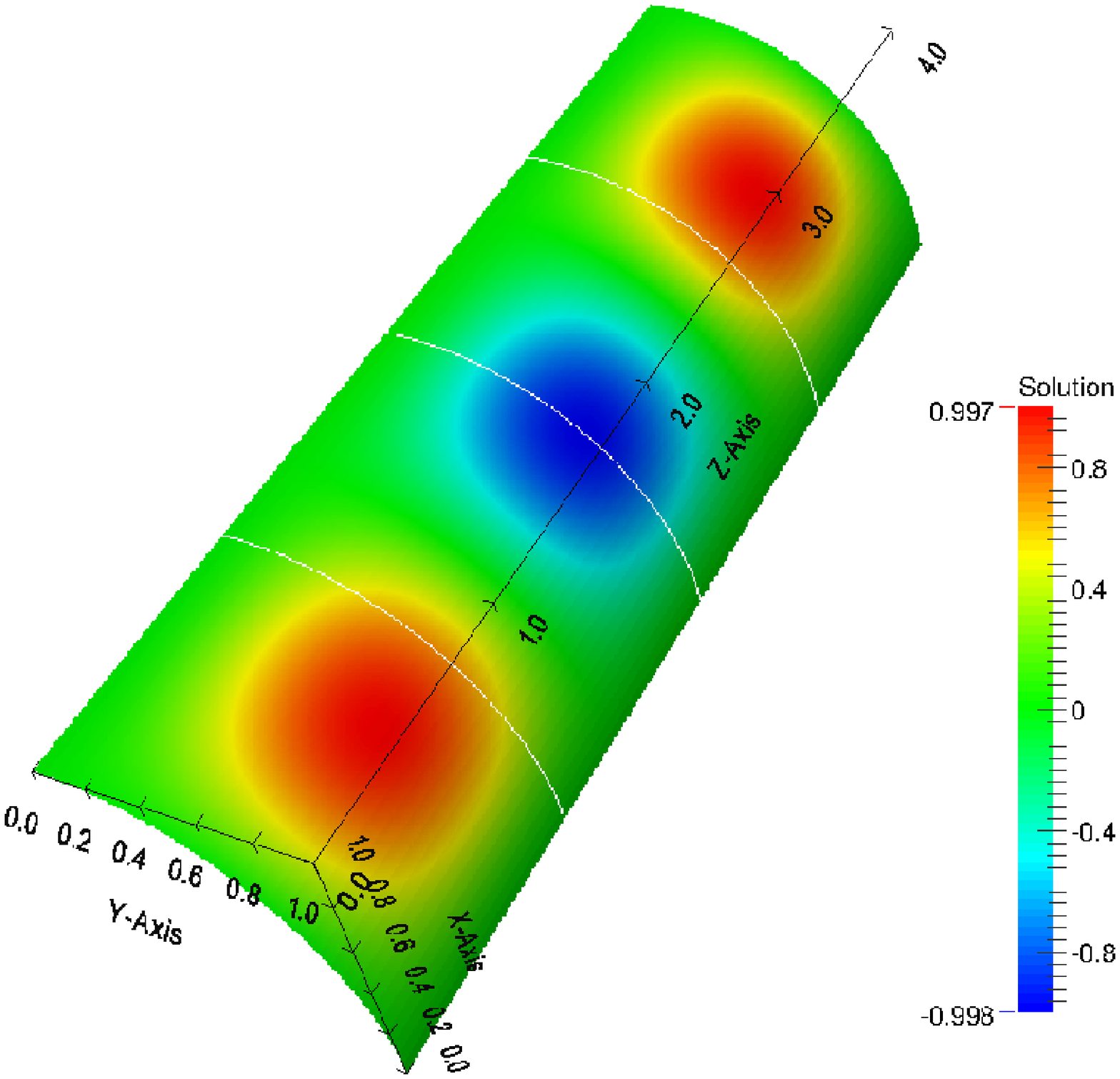}
      \caption{$2-$dimensional computational domain of a quarter cyliinder (left) consisting of four patches illustrated with different colors and the solution contours (right).}
      \label{fig:opensurface}
    \end{figure}
\begin{table}[bth!]
	\label{tab:opensurface}
	\centering 
	\begin{tabular}{|c|c|c|c|c|}	\hline
		\textbf{Method} &\textbf{SIPG} & \textbf{SSIPG1} &\textbf{SSIPG2} &\textbf{NIPG} \\ \hline
                & $0.95$& $0.99$& $0.94$ & $1.14$\\ %
		$p=2$   & $1.00$& $1.00$& $1.00$ & $1.00$ \\ %
                & $1.00$& $1.00$& $1.00$ & $1.00$\\  \hline%
		        & $2.01$& $2.05$& $2.02$ & $2.14$ \\ %
        $p=3$   & $2.00$& $2.00$& $2.00$ & $2.04$\\ %
                & $2.00$& $2.00$& $2.00$ & $2.00$\\ \hline%
                & $3.03$& $3.05$& $3.15$ & $3.43$ \\ %
		$p=4$   & $3.00$& $3.00$& $3.09$ & $3.11$\\ %
                & $3.00$& $3.00$& $3.00$ & $3.00$\\ \hline%
                & $4.07$& $4.26$& $4.04$ & $4.13$\\ %
		$p=5$   & $4.02$& $4.02$& $4.00$ & $4.01$\\ %
                & $4.00$& $4.00$& $4.00$ & $4.00$\\ \hline%
                & $5.06$& $5.13$& $5.01$ & $5.01$\\ %
		$p=6$   & $5.01$& $5.03$& $5.00$ & $5.00$\\
                & $5.00$& $5.00$& $5.00$ & $5.00$\\
		\hline
	\end{tabular}	
	\caption{Convergence rates for the open surfaace problem $\|\cdot \|_h.$} 
\end{table}
\subsection{Closed Surface}
\label{subsec:closedsurface}
We consider for the closed surface, a torus, 
$$\Omega = \{(x,y) \in (-3,3)^2, z \in (-1,1): \; r^2 = z^2 + (\sqrt{x^2+y^2}-R^2) \},$$ 
that is decomposed into 4 patches as depicted on the left-hand side of Figure~\ref{fig:closedsurface}. Each of the NURBS patches is described by the knot vectors  
$\Xi_{2}= \{0, 0, 0, 1, 1, 1\}$ and  $\Xi_{1}  = \{0, 0, 0, 0.25, 0.25, 0.50, 0.50, 0.75, 0.75, 1, 1, 1\}.$
For the surface biharmonic problem, we consider an exact solution given by $u(\phi,\theta)= \sin(3 \phi)\cos(3\theta + \phi),$ see also \cite{LarssonLarson:2013a}. We chose the exact solution $u$ and the force term $f$ such that the zero mean compatibility condition holds. 
The rate of convergence for all four $(4)$ dGIGA schemes with respect to the discrete norm $\|\cdot\|_h$ is presented in \tabref{tab:closedsurface} by successive mesh refinement for NURBS degree $2 \leq p \leq 6.$ We observe the optimal convergence rate $\mathcal{O}(h^{(p-1)})$ as theoretically predicted in Theorem~\ref{thm:dgerrorestimate}. Since the discrete norm $\|\cdot \|_h$ is indeed a norm on $H^2(\Omega,\mathcal{T}_h),$ we require at least a continuously differentiable basis function i.e. $C^{p-1}, p \geq 2$ to obtain optimal convergence rates.
    \begin{figure}[thb!]
     \centering
      \includegraphics[width=0.48\textwidth]{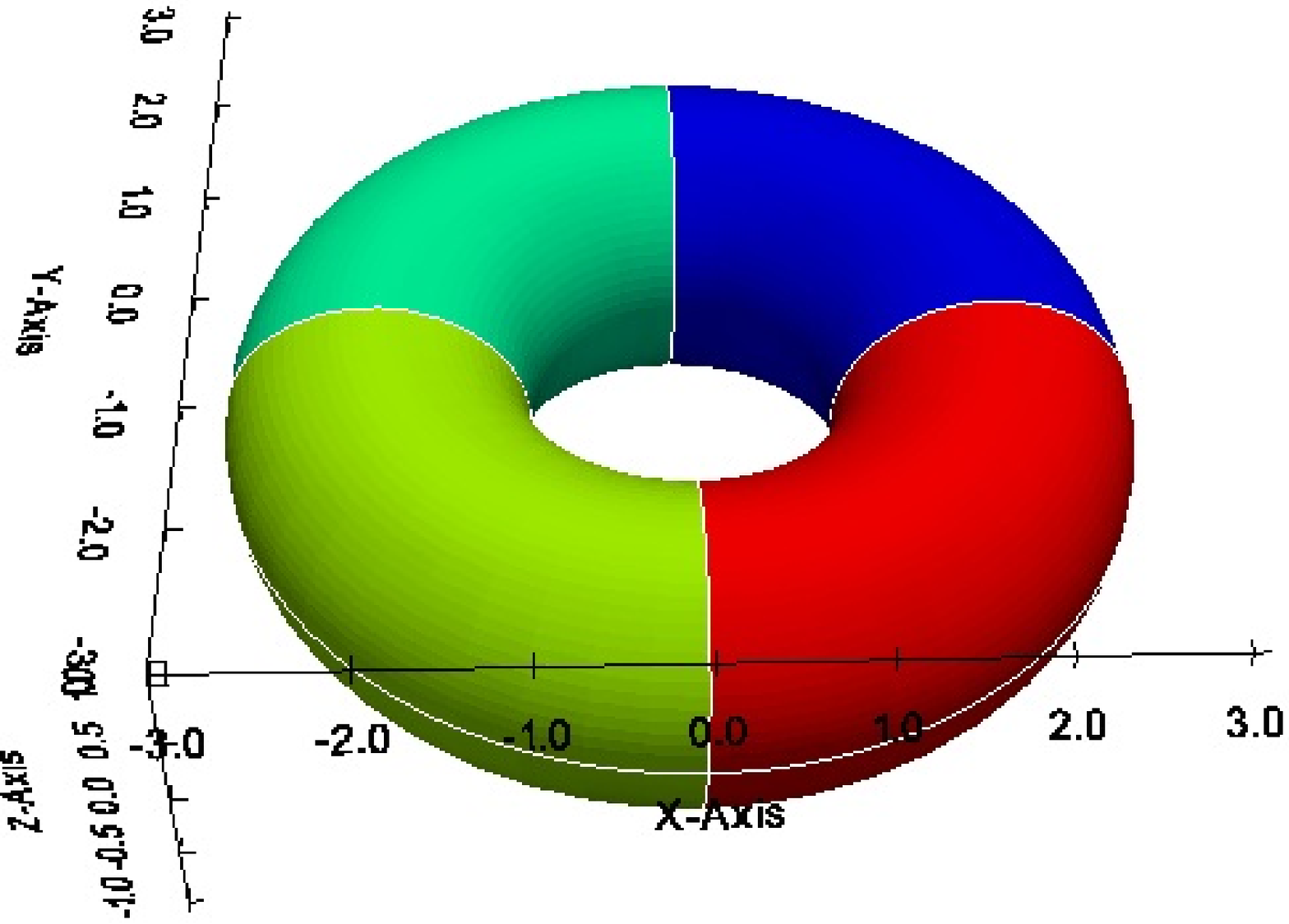}
      \includegraphics[width=0.48\textwidth]{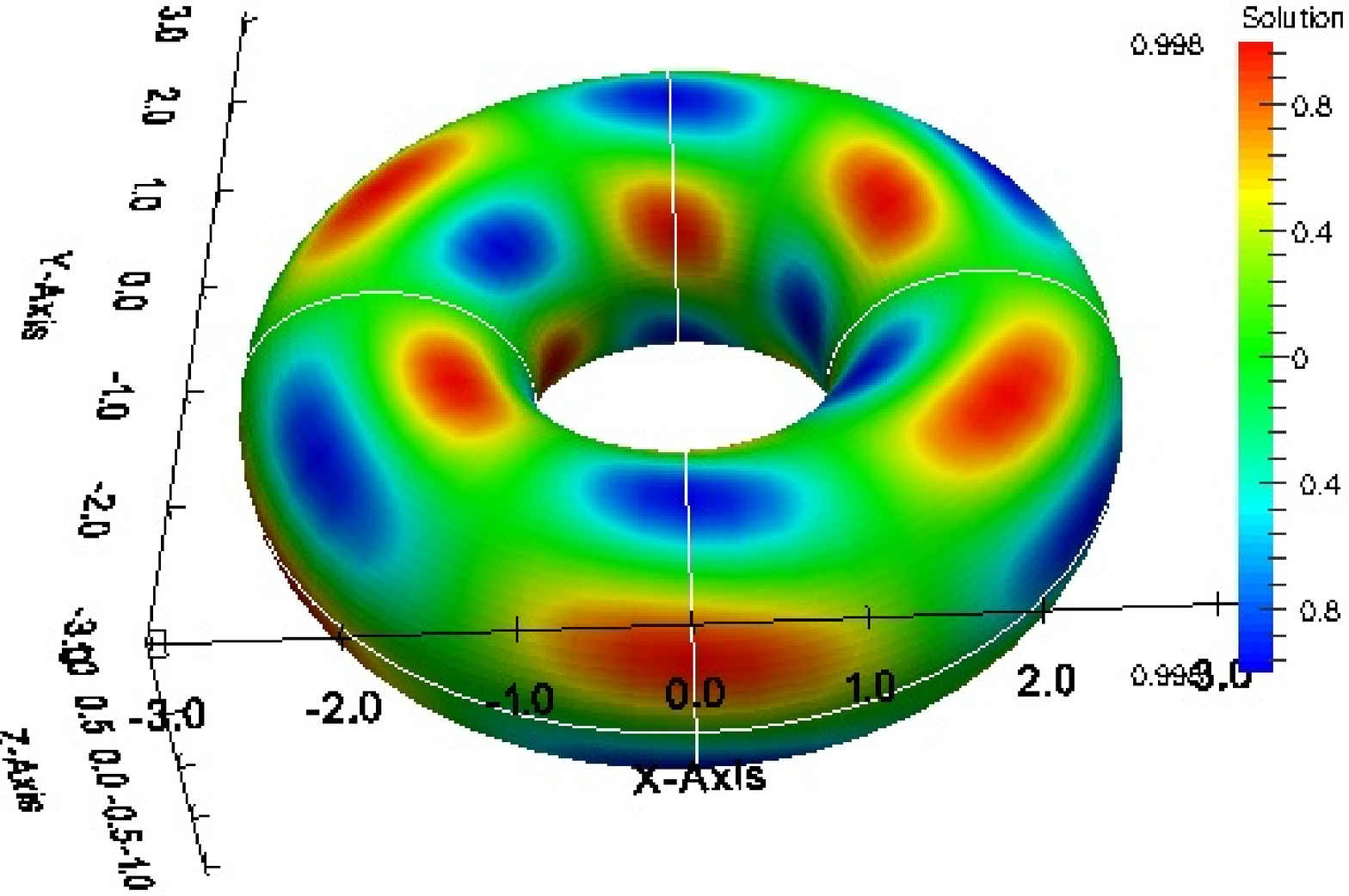}
      \caption{$2-$dimensional computational domain of a torus (left) consisting of four patches illustrated with different colors and the solution contours (right).}
      \label{fig:closedsurface}
    \end{figure}
\begin{table}[!bht]
	\label{tab:closedsurface}
	\centering 
	\begin{tabular}{|c|c|c|c|c|}	\hline
		\textbf{Method} &\textbf{SIPG} & \textbf{SSIPG1} &\textbf{SSIPG2} &\textbf{NIPG} \\ \hline
                & $1.09$& $1.12$& $1.07$ & $1.11$ \\ %
		$p=2$   & $1.00$& $1.01$& $1.00$ & $1.00$ \\ %
                & $1.00$& $1.00$& $1.00$ & $1.00$ \\  \hline%
		        & $2.02$& $2.09$& $2.05$ & $2.11$ \\ %
        $p=3$   & $2.00$& $2.00$& $2.00$ & $2.00$ \\ %
                & $2.00$& $2.00$& $2.00$ & $2.00$ \\ \hline%
                & $3.01$& $3.07$& $3.03$ & $3.10$ \\ %
		$p=4$   & $3.00$& $3.00$& $3.00$ & $3.00$ \\ %
                & $3.00$& $3.00$& $3.00$ & $3.00$ \\ \hline%
                & $4.02$& $4.07$& $4.10$ & $4.09$ \\ %
		$p=5$   & $4.00$& $4.00$& $4.00$ & $4.00$ \\ %
                & $4.00$& $4.00$& $4.00$ & $4.00$ \\ \hline%
                & $5.02$& $5.03$& $5.10$ & $5.07$ \\ %
		$p=6$   & $5.00$& $5.00$& $5.00$ & $5.00$ \\ %
                & $5.00$& $5.00$& $5.00$ & $5.00$ \\ \hline
	\end{tabular}
	\caption{Convergence rates for the closed surfaace problem $\|\cdot \|_h.$} 
\end{table}    
\section*{Conclusion}
\label{sec:Conclusion}
We have presented \textit{a priori} error estimates for the multipatch discontinuous Galerkin isogeometric  analysis (dGIGA) for the surface biharmonic problem on orientable computational domains. We assumed non-overlapping subdomains usually referred to as patches such that the solution could be discontinuous on the internal facets and applied interior penalty Galerkin techniques. We arrived at four bilinear forms namely; symmetric (SIPG), non-symmetric (NIPG), semi-symmetric 1 (SSIPG1) and semi-symmetric 2 (SSIPG2) interior penalty Galerkin methods. We showed optimal \textit{a priori} error estimates with respect to two discrete norms $\|\cdot\|_h$ and $\|\cdot\|_{h,*}$ and presented numerical results for closed and open surfaces that confirmed the analysis presented. We will extend the current results to biharmonic problems with singularities as treated for the second order elliptic PDE in \cite[Chapter 4]{Moore:2017a}.

\bibliographystyle{plain} 
\bibliography{Surface_dGBiharmonic}

\end{document}